\documentclass[12pt]{amsart}
\usepackage{amsmath,amssymb,amsthm}
\usepackage{geometry}
\usepackage{hyperref}
\usepackage{bookmark}
\usepackage[numbers,sort&compress]{natbib}
\usepackage{hyperref}
\hypersetup{colorlinks=true,linkcolor=blue,citecolor=blue,urlcolor=blue}

\numberwithin{equation}{section}

\newtheorem{theorem}{Theorem}[section]
\newtheorem{lemma}[theorem]{Lemma}
\newtheorem{proposition}[theorem]{Proposition}
\newtheorem{corollary}[theorem]{Corollary}
\theoremstyle{remark}
\newtheorem{remark}[theorem]{Remark}

\newcommand{\N}{\mathbb N}
\newcommand{\R}{\mathbb R}
\newcommand{\E}{\mathbb E}
\newcommand{\Pp}{\mathbb P}
\newcommand{\F}{\mathcal F}

\newcommand{\bPhi}{\overline\Phi}

\title{Cram\'er-Type Moderate Deviations for Engel's Series via a Martingale Approach}

\author{Shaochen WANG}
\address{School of Mathematics, South China University of Technology, Guangzhou, China}
\email{mascwang@scut.edu.cn}

\author{Guangyu YANG}
\address{School of Mathematics and Statistics, Zhengzhou University, Zhengzhou, China}
\email{guangyu@zzu.edu.cn}

\date{\today}

\begin{document}
\maketitle

\begin{abstract}
Let $x$ be uniformly distributed on $(0,1)$, and let $(q_n)_{n\geq1}$ be the digits of its Engel series expansion. We establish a Cram\'er-type moderate deviation expansion for
$(\log q_n-n)/\sqrt n$.
The proof is based on a martingale decomposition and asymptotic results for martingales. As consequences, we obtain a moderate deviation principle over the full range of scales between the central limit theorem and the law of large numbers, without the additional lower rate restriction required in several earlier works. We also derive a uniform Berry--Esseen bound of order $(\log n)/\sqrt n$.
\end{abstract}


\section{Introduction}

For every $x\in(0,1)$ there exists a unique Engel expansion
\begin{equation}\label{eq:Engel}
 x=\frac1{q_1}+\frac1{q_1q_2}+\cdots+\frac1{q_1q_2\cdots q_n}+\cdots,
\end{equation}
where the digits $q_n=q_n(x)\in\N$ satisfy $q_{n+1}\ge q_n\ge2$. Engel expansions are classical monotone series expansions of real numbers, see Oppenheim \cite{Oppenheim1972} and Galambos \cite{Galambos1976} for general background.

When $x$ is uniformly distributed on $(0,1)$, the digit sequence $(q_n)_{n\ge1}$ is a time-homogeneous Markov chain on $\{2,3,\ldots\}$ with transition probabilities
\begin{equation}\label{eq:transition}
 \Pp(q_n=k\mid q_{n-1}=j)=\frac{j-1}{k(k-1)},
 \qquad k\ge j\ge2,
\end{equation}
and initial distribution
\begin{equation}\label{eq:initial}
 \Pp(q_1=k)=\frac1{k(k-1)},\qquad k\ge2,
\end{equation}
see Erd\H{o}s, R\'enyi and Sz\"usz \cite{ErdosRenyiSzusz1958}. The study of probabilistic aspects of Engel's series goes back to Borel \cite{Borel1947}, L\'evy \cite{Levy1947}, and R\'enyi \cite{Renyi1962}. In particular (see for example L\'evy \cite{Levy1947})
\begin{equation}\label{eq:CLT-intro}
 \frac{\log q_n-n}{\sqrt n}\stackrel{d}{\rightarrow} N(0,1),
\end{equation}
where ``$\stackrel{d}{\rightarrow}$" means convergence in distribution and $N(0,1)$ is the standard normal distribution. Williams \cite{Williams1973} later related Engel's series to an $A$-process arising from R\'enyi's record problem. Large and moderate deviations were studied by Zhu \cite{Zhu2014} and Hu \cite{Hu2015} by some fine asymptotic analysis of the exponential moments of $\log q_n$. Similar results for other expansions of real numbers can be found in Fang \cite{Fang2015a,Fang2015b}, Fang and Wu \cite{FangWu2018}, and Fang, Wu and Shang \cite{FangWuShang2018}.

The purpose of this paper is to strengthen the classical asymptotic theory by proving a Cram\'er-type moderate deviation expansion. Such an expansion quantifies the relative error in the Gaussian approximation of tail probabilities and is stronger than both an ordinary moderate deviation principle and a uniform central limit estimate. Our approach is mainly based on the martingale results of Fan and Shao \cite{FanShao2024}. Related Cram\'er-type and Berry--Esseen results for martingales were developed by Grama and Haeusler \cite{GramaHaeusler2000}, Fan, Grama and Liu \cite{FanGramaLiu2013,FanGramaLiu2017}, and Fan, Grama, Liu and Shao \cite{FanGramaLiuShao2019}. Their results apply under a conditional Bernstein condition and exponential concentration of the predictable quadratic variation. Both requirements can be verified explicitly for Engel's series using suitable estimates.

The paper is organized as follows. Section \ref{sec:main} states the Cram\'er-type expansion as the main theorem. The Gaussian tail equivalence, the traditional moderate deviation principle, and the Berry--Esseen bound are then presented as corollaries. Section \ref{sec:proofs} contains some preliminary estimates and the martingale decomposition for $\log q_n$. The proofs of the main Theorem \ref{thm:main-cramer} and its corollaries are given in Section \ref{sec:proofsmain}. For convenience, we state the Cram\'er-type martingale results of Fan and Shao in the Appendix.

\section{Main results}\label{sec:main}

Let
\[
 Y_n:=\frac{\log q_n-n}{\sqrt n},
 \qquad
 \bPhi(x):=1-\Phi(x),
\]
where $\Phi$ is the standard normal distribution function. For $n\ge3$ and $x\ge0$, define
\begin{equation}\label{eq:Delta-main}
 \Delta_n(x):=\frac{x^3}{\sqrt n}
 +\frac{(1+x)\log n}{\sqrt n}.
\end{equation}

\smallskip

The constants $c,c_0,c_1,C$, and so on, appearing below may change from line to line. Our main result is the following theorem.

\begin{theorem}[Cram\'er-type moderate deviation]\label{thm:main-cramer}
There exists an absolute constant $C>0$ such that the following holds. For every positive sequence $(b_n)_{n\geq 1}$ satisfying $b_n=o(\sqrt n)$, for all sufficiently large $n$, uniformly for $0\le x\le b_n$,
\begin{align}
 \left|\log\frac{\Pp(Y_n>x)}{\bPhi(x)}\right|
 \le C\Delta_n(x),\label{eq:main-upper}
 \end{align}
 and
 \begin{align}
 \left|\log\frac{\Pp(Y_n<-x)}{\bPhi(x)}\right|
 \le C\Delta_n(x).\label{eq:main-lower}
\end{align}
The second estimate remains valid if the strict inequality $Y_n<-x$ is replaced by $Y_n\le -x$.
\end{theorem}

The proof of the above theorem is given in Section \ref{sec:proofsmain}. If $0\le x=o(n^{1/6})$, then $\Delta_n(x)\to0$ uniformly.  Using the fact that $e^u=1+O(u)$ as $u\to0$, we immediately obtain the following corollary.

\begin{corollary}\label{cor:tail-equivalence}
Uniformly for $0\le x=o(n^{1/6})$, we have
\begin{align*}
 \frac{\Pp(Y_n>x)}{\bPhi(x)}=1+o(1)
\end{align*}
 and
\begin{align*}
 \frac{\Pp(Y_n<-x)}{\bPhi(x)}=1+o(1).
\end{align*}
\end{corollary}

The Cram\'er-type moderate deviation also implies

\begin{corollary}[Moderate deviation principle]\label{cor:MDP}
Let $(a_n)_{n\geq 1}$ be any positive sequence satisfying
\begin{equation}\label{eq:an-condition}
 a_n\to\infty,
 \qquad
 \frac{a_n}{\sqrt n}\to0.
\end{equation}
Then
\[
 \frac{\log q_n-n}{a_n\sqrt n}
\]
satisfies on $\R$ a moderate deviation principle with speed $a_n^2$ and good rate function
\[
 I(x)=\frac{x^2}{2},\quad x\in\R.
\]
Equivalently, for every Borel set $B\subset\R$,
\begin{align*}
 -\inf_{x\in B^\circ}\frac{x^2}{2}
 &\le \liminf_{n\to\infty}\frac1{a_n^2}
 \log\Pp\left(\frac{\log q_n-n}{a_n\sqrt n}\in B\right)\\
 &\le \limsup_{n\to\infty}\frac1{a_n^2}
 \log\Pp\left(\frac{\log q_n-n}{a_n\sqrt n}\in B\right)
 \le-\inf_{x\in\overline B}\frac{x^2}{2}.
\end{align*}
\end{corollary}

\begin{remark}
The conditions in \eqref{eq:an-condition}, namely $a_n\to\infty$ and $a_n/\sqrt n\to0$, separate the moderate deviation regime from the central limit regime. In Hu \cite{Hu2015}, Fang \cite{Fang2015a,Fang2015b}, and Fang, Wu and Shang \cite{FangWuShang2018}, the additional condition
$a_n/\sqrt{n\log a_n}\to\infty$ is imposed because of technical restrictions. The martingale approach employed here demonstrates that this condition is not needed and also applies to Sylvester's series and Cantor's products, see Zhu \cite{Zhu2014}.
\end{remark}

\begin{corollary}[Berry--Esseen bound]\label{cor:BE}
There exists an absolute constant $C>0$ such that, for all $n\ge3$,
\begin{equation}\label{eq:BE-main}
 \sup_{x\in\R}
 \left|\Pp(Y_n\le x)-\Phi(x)\right|
 \le C\frac{\log n}{\sqrt n}.
\end{equation}
In particular, the central limit theorem \eqref{eq:CLT-intro} follows.
\end{corollary}

The bound in \eqref{eq:BE-main} has order $(\log n)/\sqrt n$. It is natural to conjecture that the special transition structure of Engel's series may permit the sharper order $n^{-1/2}$ as in the i.i.d. case. The general martingale estimate in Lemma~\ref{lem:FS-BE} contains the term $\epsilon_n|\log\epsilon_n|$, which is known to be unavoidable for general martingales under comparable assumptions. Thus an improvement to $C/\sqrt n$ may require another argument specific to Engel's series.

\begin{remark}
  A historical point deserves emphasis here. According to Erd\H{o}s, R\'enyi and Sz\"usz \cite{ErdosRenyiSzusz1958} (see also Zhu \cite{Zhu2014}, Remark 1.4), L\'evy \cite{Levy1947} showed that
  the quantities $\log(q_n/q_{n-1})$ are asymptotically close to i.i.d. exponential random variables with mean $1$.
  But such an approximation is not enough to study fine tail probabilities, such as large deviations, moderate deviations and so on.
\end{remark}

\section{Preliminary lemmas and martingale decomposition}\label{sec:proofs}

In this section, we establish preliminary estimates for the exact conditional tails and moments of $\log q_n$, together with the corresponding martingale decomposition.

\subsection{Conditional tails and moments}

Let $\F_0=\{\varnothing,\Omega\}$ and $\F_n=\sigma(q_1,\ldots,q_n)$ for $n\ge1$. Put $q_0=1$ and
\begin{equation}\label{eq:rn}
 r_n:=\frac{q_n}{q_{n-1}},\qquad n\ge1.
\end{equation}
Then
\begin{equation}\label{eq:logsum}
 \log q_n=\sum_{k=1}^n\log r_k.
\end{equation}

\begin{lemma}\label{lem:tail}
For every $n\ge2$, $j\ge2$, and $t\ge0$,
\begin{equation}\label{eq:tail-exact}
 \Pp(\log r_n>t\mid q_{n-1}=j)
 =\Pp(q_n\ge\lfloor je^t\rfloor+1\mid q_{n-1}=j)
 =\frac{j-1}{\lfloor je^t\rfloor}.
\end{equation}
Consequently,
\begin{equation}\label{eq:tail-bounds}
 \left(1-\frac1j\right)e^{-t}
 \le \Pp(\log r_n>t\mid q_{n-1}=j)
 \le\frac{j-1}{je^t-1}
 \le2e^{-t}.
\end{equation}
Moreover, there is an absolute constant $C>0$ such that
\begin{equation}\label{eq:tail-error}
 \left|\Pp(\log r_n>t\mid q_{n-1}=j)-e^{-t}\right|
 \le\frac{C}{j}e^{-t}.
\end{equation}
\end{lemma}

\begin{proof}
For every integer $m\ge j\ge 2$,
\[
 \Pp(q_n\ge m\mid q_{n-1}=j)
 =(j-1)\sum_{k=m}^\infty\frac1{k(k-1)}
 =\frac{j-1}{m-1}.
\]
Under the conditioning $q_{n-1}=j$,
\[
 \{\log r_n>t\}=\{q_n>je^t\}
 =\{q_n\ge\lfloor je^t\rfloor+1\},
\]
which proves \eqref{eq:tail-exact}. The first two bounds in \eqref{eq:tail-bounds} follow from
$\lfloor je^t\rfloor\le je^t$ and $\lfloor je^t\rfloor\ge je^t-1$. Since $je^t-1\ge je^t/2$, the last bound follows as well.

For \eqref{eq:tail-error}, write
\[
 \frac{j-1}{\lfloor je^t\rfloor}-e^{-t}
 =e^{-t}\left\{\frac{(j-1)je^t}{j\lfloor je^t\rfloor}-1\right\}.
\]
Furthermore,
\[
 \left|\frac{je^t}{\lfloor je^t\rfloor}-1\right|
 \le\frac1{\lfloor je^t\rfloor}
 \le\frac1{je^t-1}
 \le\frac1{j-1}.
\]
It follows that
\[
 \left|\frac{(j-1)je^t}{j\lfloor je^t\rfloor}-1\right|
 \le\frac1j+\frac{j-1}{j}
 \left|\frac{je^t}{\lfloor je^t\rfloor}-1\right|
 \le\frac{C}{j},
\]
and the proof is complete.
\end{proof}

\begin{lemma}[Conditional moments]\label{lem:moments}
There exists an absolute constant $C>0$ such that, for every $n\ge2$ and $j\ge2$,
\begin{align}
 \left|\E(\log r_n\mid q_{n-1}=j)-1\right|\le\frac{C}{j},\label{eq:first-moment}\\
 \left|\E((\log r_n)^2\mid q_{n-1}=j)-2\right|\le\frac{C}{j},\label{eq:second-moment}\\
 \E((\log r_n)^\ell\mid q_{n-1}=j)\le2\ell!,
 \qquad \ell\ge1.\label{eq:all-moments}
\end{align}
Moreover, $\E(\log q_1)^\ell<\infty$ for every integer $\ell\ge1$.
\end{lemma}

\begin{proof}
Let $X_{n,j}$ denote $\log r_n$ under the conditioning $q_{n-1}=j$. Since $X_{n,j}\ge0$, the tail-integral formula gives
\[
 \E X_{n,j}^\ell
 =\ell\int_0^\infty t^{\ell-1}\Pp(X_{n,j}>t)\,dt.
\]
By \eqref{eq:tail-error},
\[
 \left|\E X_{n,j}^\ell-\ell!\right|
 \le\frac{C\ell!}{j}.
\]
Taking $\ell=1,2$ proves \eqref{eq:first-moment} and \eqref{eq:second-moment}. From \eqref{eq:tail-bounds},
\[
 \E X_{n,j}^\ell
 \le2\ell\int_0^\infty t^{\ell-1}e^{-t}\,dt
 =2\ell!,
\]
which proves \eqref{eq:all-moments}. Finally,
\[
 \E(\log q_1)^\ell
 =\sum_{k=2}^\infty\frac{(\log k)^\ell}{k(k-1)}<\infty.
\]
\end{proof}

\begin{lemma}[Negative moments]\label{lem:negative}
There exists an absolute constant $\rho\in(0,1)$ such that, for every integer $p\ge1$, $n\ge2$, and $j\ge2$,
\begin{equation}\label{eq:negative-conditional}
 \E(q_n^{-p}\mid q_{n-1}=j)\le\rho j^{-p}.
\end{equation}
Consequently,
\begin{equation}\label{eq:negative-unconditional}
 \E q_n^{-p}\le2^{-p}\rho^{n-1},
 \qquad n\ge1,
 \quad p\ge1.
\end{equation}
Here, we may take $\rho=3/4$ for convenience.
\end{lemma}

\begin{proof}
Under the conditioning $q_{n-1}=j$, put $Z_j=j/q_n$. Since $0<Z_j\le1$, $Z_j^p\le Z_j$ for $p\ge1$. Therefore
\begin{equation}\label{eq:negative-via-first}
 \E(q_n^{-p}\mid q_{n-1}=j)
 =j^{-p}\E Z_j^p
 \le j^{-p}\E Z_j.
\end{equation}
By \eqref{eq:transition},
\begin{align*}
 \E Z_j
 &=j(j-1)\sum_{k=j}^\infty\frac1{k^2(k-1)}\\
 &=j(j-1)\sum_{k=j}^\infty
 \left(\frac1{k-1}-\frac1k-\frac1{k^2}\right)\\
 &=j-j(j-1)\sum_{k=j}^\infty\frac1{k^2}.
\end{align*}

For $j\ge2$, convexity of $x\mapsto x^{-2}$ gives
\[
 \sum_{k=j}^\infty\frac1{k^2}
 \ge\int_j^\infty x^{-2}\,dx+\frac1{2j^2}
 =\frac1j+\frac1{2j^2}.
\]
Hence
\[
 \E Z_j\le\frac12+\frac1{2j}\le\frac34.
\]
 Thus \eqref{eq:negative-conditional} follows from \eqref{eq:negative-via-first}. Iteration and $q_1\ge2$ yield \eqref{eq:negative-unconditional}.
\end{proof}

\subsection{Martingale decomposition and conditional Bernstein condition}

For $j\ge2$, define
\[
 \mu(j):=\E(\log r_k\mid q_{k-1}=j),
 \qquad
 \varepsilon(j):=1-\mu(j),
\]
where time homogeneity makes the definition independent of $k\ge2$. Put
\[
 \mu_1:=\E\log q_1,
 \qquad
 D_1:=\log q_1-\mu_1,
 \qquad
 D_k:=\log r_k-\mu(q_{k-1}),\quad k\ge2.
\]
Then $(D_k,\F_k)_{k\ge1}$ is a martingale difference sequence. Define
\begin{equation}\label{eq:mart-defs}
 M_n:=\sum_{k=1}^nD_k,
 \qquad
 R_n:=\sum_{k=2}^n\varepsilon(q_{k-1}),
 \qquad
 c_1:=\mu_1-1.
\end{equation}
By \eqref{eq:logsum},
\begin{equation}\label{eq:decomposition}
 \log q_n-n=M_n-R_n+c_1.
\end{equation}
Lemma~\ref{lem:moments} gives
\begin{equation}\label{eq:epsilon-bound}
 |\varepsilon(j)|\le\frac{C}{j},\qquad j\ge2.
\end{equation}

The next lemma gives exponential control of the correction term $R_n$.

\begin{lemma}\label{lem:drift}
Let
\begin{equation}\label{eq:Ainfty}
 A_\infty:=\sum_{m=1}^\infty q_m^{-1}.
\end{equation}
Then $A_\infty<\infty$ almost surely and $A_\infty\in L^p$, $1\le p<\infty$. Moreover, there exist constants $C_0,\lambda_0>0$ such that
\begin{equation}\label{eq:exp-moment-A}
 \E e^{\lambda_0A_\infty}<\infty,
 \qquad
 \Pp(A_\infty>u)\le C_0e^{-\lambda_0u},\quad u\ge0.
\end{equation}
In particular,
\begin{equation}\label{eq:Rn-dominated}
 |R_n|\le CA_\infty,
 \qquad n\ge1.
\end{equation}
\end{lemma}

\begin{proof}
By Minkowski's inequality and Lemma~\ref{lem:negative}, for every $p\ge1$,
\[
 \|A_\infty\|_p
 \le\sum_{m=1}^\infty\|q_m^{-1}\|_p
 \le\frac12\sum_{m=1}^\infty\rho^{(m-1)/p}
 =\frac1{2(1-\rho^{1/p})}.
\]
Since $\rho\in(0,1)$, $1-\rho^{1/p}\ge c/p$, and therefore
\begin{equation}\label{eq:Ap-bound}
 \|A_\infty\|_p\le Cp,
 \qquad p\ge1.
\end{equation}
This proves the asserted almost-sure and $L^p$ convergence. For integer $p\ge1$,
\[
 \E A_\infty^p\le(Cp)^p\le(Ce)^pp!.
\]
Consequently, $\E e^{\lambda A_\infty}<\infty$ for sufficiently small $\lambda>0$, and the tail bound follows from Chernoff's inequality. Finally, \eqref{eq:Rn-dominated} follows from \eqref{eq:epsilon-bound}.
\end{proof}

Define
\begin{equation}\label{eq:snVn}
 s_n^2:=\E M_n^2=\sum_{k=1}^n\E D_k^2,
 \qquad
 V_n^2:=\sum_{k=1}^n\E(D_k^2\mid\F_{k-1}).
\end{equation}
The next lemma gives variance estimates for the martingale $M_n$.

\begin{lemma}\label{lem:variance}
There exist absolute constants $c_0,C_0,C_1,C_2>0$ such that, for every $n\ge1$,
\begin{align}
 |V_n^2-n|&\le C_0(1+A_\infty)\quad\text{a.s.},\label{eq:Vn-n}\\
 |s_n^2-n|&\le C_1,\label{eq:sn-n}\\
 c_0\le\E(D_k^2\mid\F_{k-1})&\le C_2
 \quad\text{a.s. for every }k\ge1.\label{eq:var-bounds}
\end{align}
In particular, $s_n^2\asymp n$.
\end{lemma}

\begin{proof}
For $k\ge2$, Lemma~\ref{lem:moments} gives, conditionally on $q_{k-1}=j$,
\[
 \E(\log r_k\mid q_{k-1}=j)=1+O(j^{-1}),
 \qquad
 \E((\log r_k)^2\mid q_{k-1}=j)=2+O(j^{-1}),
\]
with constants independent of $k$ and $j$. Therefore
\begin{equation}\label{eq:cond-var}
 \E(D_k^2\mid\F_{k-1})
 =1+O(q_{k-1}^{-1}),
 \qquad k\ge2.
\end{equation}
Summation gives
\[
 |V_n^2-n|
 \le|\E D_1^2-1|+C\sum_{k=2}^nq_{k-1}^{-1}
 \le C(1+A_\infty),
\]
which proves \eqref{eq:Vn-n}. Taking expectations and using Lemma~\ref{lem:drift} gives \eqref{eq:sn-n}, and hence $s_n^2\asymp n$.

The upper bound in \eqref{eq:var-bounds} follows from Lemma~\ref{lem:moments}. By \eqref{eq:cond-var}, the conditional variance is at least $1/2$ for all sufficiently large states $j$. For each of the finitely many remaining states, the conditional law in \eqref{eq:transition} is non-degenerate, so its variance is strictly positive. Taking a finite minimum gives the lower bound for $k\ge2$. The case $k=1$ is also non-degenerate.
\end{proof}

With the above estimates, we can now verify the conditional Bernstein condition, which is
crucial for proving the Cram\'er-type moderate deviation results for martingales.

\begin{lemma}[Conditional Bernstein condition]\label{lem:bernstein}
There exists an absolute constant $H>0$ such that, for every integer $m\ge3$ and every $k\ge1$,
\begin{equation}\label{eq:Bernstein}
 \E(|D_k|^m\mid\F_{k-1})
 \le\frac{m!}{2}H^{m-2}\E(D_k^2\mid\F_{k-1})
 \quad\text{a.s.}
\end{equation}
\end{lemma}

\begin{proof}
For $k\ge2$, write $X=\log r_k$ and $\mu=\E(X\mid\F_{k-1})$. Lemma~\ref{lem:moments} gives $|\mu|\le C$ and $\E(X^m\mid\F_{k-1})\le2m!$. Hence
\[
 \E(|D_k|^m\mid\F_{k-1})
 \le2^{m-1}\{\E(X^m\mid\F_{k-1})+|\mu|^m\}
 \le C_1^m m!.
\]
By Lemma~\ref{lem:variance}, $\E(D_k^2\mid\F_{k-1})\ge c_0$. Choosing $H$ sufficiently large gives \eqref{eq:Bernstein}. For $k=1$, the bound $\Pp(\log q_1>t)\le2e^{-t}$ gives the same factorial moment estimate after centering, and the same conclusion follows after increasing $H$ if necessary.
\end{proof}

\section{Proofs of Theorem \ref{thm:main-cramer} and its corollaries}\label{sec:proofsmain}

To apply the Cram\'er-type moderate deviation results for martingales, we need the following lemma, which corresponds to condition \eqref{eq:FS-A2} in Appendix.

\begin{lemma}[Concentration of the predictable quadratic variation]\label{lem:qv}
There exist constants $C,c>0$ such that, for every $n\ge2$ and $u>0$,
\begin{equation}\label{eq:qv-concentration}
 \Pp\left(\left|\frac{V_n^2}{s_n^2}-1\right|>u\right)
 \le C e^{-cnu^2}.
\end{equation}
\end{lemma}

\begin{proof}
By Lemma~\ref{lem:variance}, $s_n^2=n+O(1)$ and
\[
 \left|\frac{V_n^2}{s_n^2}-1\right|
 =\frac{|V_n^2-s_n^2|}{s_n^2}
 \le \frac{|V_n^2-n|+|s_n^2-n|}{s_n^2}
 \le \frac{C(1+A_\infty)}{n}.
\]
Consequently,  write $C_*$ for the constant in the preceding deterministic bound and set $C_0:=2C_*$.  If $u\ge C_0/n$ and
\[
 \left|\frac{V_n^2}{s_n^2}-1\right|>u,
\]
then $C_*(1+A_\infty)/n>u$, and hence
\[
 A_\infty>\frac{nu}{C_*}-1
 \ge \frac{nu}{2C_*}.
\]
Therefore, by Lemma~\ref{lem:drift},
\begin{equation}\label{eq:qv-linear-tail}
 \Pp\left(\left|\frac{V_n^2}{s_n^2}-1\right|>u\right)
 \le \Pp\left(A_\infty>\frac{nu}{2C_*}\right)
 \le C_1e^{-c_1nu},
\end{equation}
where one may take $c_1=\lambda_0/(2C_*)$.

The upper variance bound in \eqref{eq:var-bounds}, together with $s_n^2\asymp n$, gives a deterministic constant $B<\infty$ such that
\[
 0\le \frac{V_n^2}{s_n^2}\le B
 \qquad\text{a.s. for all }n.
\]
Thus the probability in \eqref{eq:qv-concentration} is zero when $u>B+1$. For {$C_0/n\le u\le B+1$, the inequality $u^2\le(B+1)u$ gives $nu\ge nu^2/(B+1)$}. Hence \eqref{eq:qv-linear-tail} yields
\[
 \Pp\left(\left|\frac{V_n^2}{s_n^2}-1\right|>u\right)
 \le C_1\exp\left\{-\frac{c_1}{B+1}nu^2\right\}.
\]
Finally, if {$0<u<C_0/n$, then $nu^2<C_0^2/n\le C_0^2/2$ because $n\ge2$. Thus $e^{-cnu^2}\ge e^{-cC_0^2/2}$. Since the probability on the left is at most $1$, the desired estimate follows by choosing the leading constant $C\ge e^{cC_0^2/2}$}. This proves \eqref{eq:qv-concentration} for every $u>0$.
\end{proof}

We are now ready to derive the Cram\'er-type moderate deviation and Berry--Esseen bounds for the martingale
$M_n/s_n$.

\begin{proposition}\label{prop:martingale-results}
Let $Z_n=M_n/s_n$. There exists an absolute constant $C>0$ such that, for every positive sequence $b_n=o(\sqrt n)$, for all sufficiently large $n$, uniformly for $0\le x\le b_n$,
\begin{align}
 \left|\log\frac{\Pp(Z_n>x)}{\bPhi(x)}\right|
 \le C\Delta_n(x),\label{eq:mart-cramer-upper}
 \end{align}
 and
 \begin{align}
 \left|\log\frac{\Pp(Z_n<-x)}{\bPhi(x)}\right|
 \le C\Delta_n(x).\label{eq:mart-cramer-lower}
\end{align}
Moreover,
\begin{equation}\label{eq:mart-BE}
 \sup_{x\in\R}|\Pp(Z_n\le x)-\Phi(x)|
 \le C\frac{\log n}{\sqrt n}.
\end{equation}
\end{proposition}

\begin{proof}
For each $n$, set
\[
 \xi_{k,n}:=\frac{D_k}{s_n},\qquad 1\le k\le n,
 \qquad
 Z_n=\sum_{k=1}^n\xi_{k,n}.
\]
Then $(\xi_{k,n},\F_k)_{1\le k\le n}$ is a martingale-difference array and
\[
 \sum_{k=1}^n\E\xi_{k,n}^2=1,
 \qquad
 \langle Z\rangle_n
 :=\sum_{k=1}^n\E(\xi_{k,n}^2\mid\F_{k-1})
 =\frac{V_n^2}{s_n^2}.
\]
By Lemma~\ref{lem:bernstein}, for every integer $m\ge3$,
\[
 \big|\E(\xi_{k,n}^m\mid\F_{k-1})\big|
 \le \E(|\xi_{k,n}|^m\mid\F_{k-1})
 \le \frac{m!}{2}\epsilon_n^{m-2}
 \E(\xi_{k,n}^2\mid\F_{k-1}),
 \qquad
 \epsilon_n:=\frac{H}{s_n}.
\]
For $m=2$ the corresponding inequality is an equality. Since $s_n^2\asymp n$, we have $\epsilon_n\le Cn^{-1/2}$ and, for all sufficiently large $n$, $\epsilon_n\le1/2$. Thus condition \eqref{eq:FS-A1} in Appendix~\ref{app:FanShao} holds.

By Lemma~\ref{lem:qv}, after increasing a constant if necessary,
\[
 \Pp(|\langle Z\rangle_n-1|>u)
 \le C\exp\{-u^2\delta_n^{-2}\},
 \qquad u>0,
 \qquad
 \delta_n:=C_0n^{-1/2}.
\]
For all sufficiently large $n$, $\delta_n\le1/2$, so condition \eqref{eq:FS-A2} also holds.

Lemma~\ref{lem:FS-cramer} now gives, uniformly for
\[
 0\le x=o\!\left(\min\{\epsilon_n^{-1},\delta_n^{-1}\}\right)
 =o(\sqrt n),
\]
\[
 \left|\log\frac{\Pp(Z_n>x)}{\bPhi(x)}\right|
 \le C\left[
 x^3(\epsilon_n+\delta_n)
 +(1+x)\{\delta_n|\log\delta_n|
 +\epsilon_n|\log\epsilon_n|\}
 \right].
\]
Since $\epsilon_n,\delta_n=O(n^{-1/2})$, the right-hand side is bounded by $C\Delta_n(x)$, proving \eqref{eq:mart-cramer-upper}. Applying the same lemma to the array $(-\xi_{k,n},\F_k)$ proves \eqref{eq:mart-cramer-lower}.

Finally, Lemma~\ref{lem:FS-BE} yields
\[
 \sup_{x\in\R}|\Pp(Z_n\le x)-\Phi(x)|
 \le C\{\epsilon_n|\log\epsilon_n|
 +\delta_n|\log\delta_n|\}
 \le C\frac{\log n}{\sqrt n},
\]
which is \eqref{eq:mart-BE}. Enlarging $C$ covers the finitely many values of $n$ excluded above.
\end{proof}

\smallskip

We now prove the main theorem.

\begin{proof}[Proof of Theorem~\ref{thm:main-cramer}]
The decomposition \eqref{eq:decomposition} can be written as
\begin{equation}\label{eq:Y-decomp}
 Y_n=\alpha_n Z_n+W_n,
 \qquad
 \alpha_n:=\frac{s_n}{\sqrt n},
 \qquad
 W_n:=\frac{-R_n+c_1}{\sqrt n}.
\end{equation}
Since $s_n^2=n+O(1)$,
\begin{equation}\label{eq:alpha}
 \alpha_n=1+O(n^{-1}).
\end{equation}
Moreover, $|R_n|\le CA_\infty$ by Lemma~\ref{lem:drift}. Hence there are constants $C,c>0$ such that
\begin{equation}\label{eq:W-tail}
 \Pp(|W_n|>t)\le Ce^{-ct\sqrt n},
 \qquad t\ge Cn^{-1/2}.
\end{equation}

Fix a sequence $b_n=o(\sqrt n)$ and consider $0\le x\le b_n$. Let
\begin{equation}\label{eq:h-def}
 L_n(x):=K(1+\log n+x^2),
 \qquad
 h_n(x):=\frac{L_n(x)}{\sqrt n},
\end{equation}
where the absolute constant $K$ will be chosen below. We divide the argument into a central range and a moderate-tail range.

\smallskip
\noindent\emph{Central range: $x\le2h_n(x)$.}
Since $b_n/\sqrt n\to0$, uniformly for $x\le b_n$,
\[
 \frac{2Kx^2}{\sqrt n}\le o(1)x.
\]
Thus, for all sufficiently large $n$, the inequality $x\le2h_n(x)$ implies
\begin{equation}\label{eq:small-x-size}
 x\le C\frac{1+\log n}{\sqrt n}.
\end{equation}
We first transfer the Berry--Esseen estimate from $Z_n$ to $Y_n$. From \eqref{eq:alpha},
\[
 \sup_{u\in\R}|\Phi(u/\alpha_n)-\Phi(u)|
 \le C|\alpha_n-1|
 \le Cn^{-1},
\]
where the first bound follows from the boundedness of $v\mapsto |v|\phi(v)$. Therefore Proposition~\ref{prop:martingale-results} gives
\begin{equation}\label{eq:scaled-mart-BE}
 \sup_{u\in\R}|\Pp(\alpha_nZ_n\le u)-\Phi(u)|
 \le C\frac{\log n}{\sqrt n}.
\end{equation}
Note that, for arbitrary random variables $U,W$ and every $\eta>0$,
\begin{equation}\label{eq:smoothing}
 \sup_{u\in\R}|\Pp(U+W\le u)-\Pp(U\le u)|
 \le \sup_{u\in\R}\Pp(u-\eta<U\le u+\eta)+\Pp(|W|>\eta).
\end{equation}
Take $U=\alpha_nZ_n$, $W=W_n$, and {$\eta=K_0(\log n)/\sqrt n$}, where $K_0$ is large. {More precisely, $K_0$ is fixed so that $\eta\ge Cn^{-1/2}$, as required in \eqref{eq:W-tail}, and $cK_0>1$, after a further increase of $K_0$, the term $n^{-cK_0}$ below is bounded by a constant multiple of $(\log n)/\sqrt n$ for every sufficiently large $n$.} The interval probability in \eqref{eq:smoothing} is at most
\[
 \Phi(u+\eta)-\Phi(u-\eta)
 +2\sup_v|\Pp(U\le v)-\Phi(v)|
 \le C\eta+C\frac{\log n}{\sqrt n},
\]
and \eqref{eq:W-tail} gives $\Pp(|W_n|>\eta)\le Cn^{-cK_0}$. Consequently,
\begin{equation}\label{eq:local-BE}
 \sup_{u\in\R}|\Pp(Y_n\le u)-\Phi(u)|
 \le C\frac{\log n}{\sqrt n}.
\end{equation}
By \eqref{eq:small-x-size}, $x=o(1)$ in the present range, so $\bPhi(x)$ is bounded below by an absolute positive constant. The estimate \eqref{eq:local-BE} also controls left limits of the distribution function, because $\Phi$ is continuous. Hence it applies to both $\Pp(Y_n\le -x)$ and $\Pp(Y_n<-x)$, and gives
\[
 \frac{\Pp(Y_n>x)}{\bPhi(x)}=1+O\!\left(\frac{\log n}{\sqrt n}\right),
 \qquad
 \frac{\Pp(Y_n<-x)}{\bPhi(x)}=1+O\!\left(\frac{\log n}{\sqrt n}\right).
\]
Since $\Delta_n(x)\ge(\log n)/\sqrt n$, the elementary inequality $|\log(1+v)|\le2|v|$ for $|v|\le1/2$ proves \eqref{eq:main-upper} and \eqref{eq:main-lower} in the central range.

\smallskip
\noindent\emph{Moderate-tail range: $x>2h_n(x)$.}
Define
\[
 z_{n,-}:=\frac{x-h_n(x)}{\alpha_n},
 \qquad
 z_{n,+}:=\frac{x+h_n(x)}{\alpha_n},
 \qquad
 p_n(x):=\Pp(|W_n|>h_n(x)).
\]
Then $z_{n,-}>0$, and the inclusions obtained from \eqref{eq:Y-decomp} give
\begin{equation}\label{eq:sandwich}
 \Pp(Z_n>z_{n,+})-p_n(x)
 \le\Pp(Y_n>x)
 \le\Pp(Z_n>z_{n,-})+p_n(x).
\end{equation}
Since $\alpha_n=1+O(n^{-1})$,
\begin{equation}\label{eq:z-shift}
 |z_{n,\pm}-x|
 \le C\left\{h_n(x)+\frac{x}{n}\right\}
 \le C\left\{\frac{1+\log n+x^2}{\sqrt n}+\frac{x}{n}\right\}.
\end{equation}
Furthermore, $h_n(x)<x/2$ and $\alpha_n\to1$, so, for all sufficiently large $n$,
\begin{equation}\label{eq:z-comparable}
\frac{x}{3}\le z_{n,-}\le z_{n,+}\le2x,
 \qquad z_{n,\pm}=o(\sqrt n).
\end{equation}

Let $\phi$ denote the standard normal density. The Gaussian hazard function satisfies
\[
 \frac{\phi(t)}{\bPhi(t)}\le C(1+t),\qquad t\ge0,
\]
by Mills' inequalities. Integrating the logarithmic derivative of $\bPhi$ between $x$ and $y$ gives
\begin{equation}\label{eq:normal-perturb}
 \left|\log\frac{\bPhi(y)}{\bPhi(x)}\right|
 \le C(1+x+|y-x|)|y-x|,
 \qquad x,y\ge0.
\end{equation}
Combining \eqref{eq:z-shift}, \eqref{eq:z-comparable}, and the elementary inequality $x^2\le1+x^3$ {for $x\geq0$}, we obtain
{the following estimates. Indeed, \eqref{eq:z-comparable} implies $|z_{n,\pm}-x|\le x$ and hence
\begin{align*}
 \left|\log\frac{\bPhi(z_{n,\pm})}{\bPhi(x)}\right|
 &\le C(1+x)|z_{n,\pm}-x|\\
 &\le C\left\{
 \frac{(1+x)(1+\log n+x^2)}{\sqrt n}
 +\frac{(1+x)x}{n}\right\}.
\end{align*}
Now $(1+x)x^2=x^2+x^3\le1+2x^3$, while $1+\log n\le C\log n$ and $n^{-1}\le(\log n)/\sqrt n$ for $n\ge3$.  The last display is therefore bounded by $C\Delta_n(x)$.  Moreover, $z_{n,\pm}\le2x$ in \eqref{eq:z-comparable}, so
\[
 \Delta_n(z_{n,\pm})
 \le \frac{8x^3}{\sqrt n}
 +\frac{(1+2x)\log n}{\sqrt n}
 \le C\Delta_n(x).
\]
Hence}
\begin{equation}\label{eq:shift-control}
 \left|\log\frac{\bPhi(z_{n,\pm})}{\bPhi(x)}\right|
 \le C\Delta_n(x),
 \qquad
 \Delta_n(z_{n,\pm})\le C\Delta_n(x).
\end{equation}
Proposition~\ref{prop:martingale-results} is applicable at $z_{n,\pm}$ because of \eqref{eq:z-comparable}. Together with \eqref{eq:shift-control}, it yields
\begin{equation}\label{eq:shifted-mart-tail}
 \bPhi(x)e^{-C\Delta_n(x)}
 \le\Pp(Z_n>z_{n,\pm})
 \le\bPhi(x)e^{C\Delta_n(x)}.
\end{equation}

It remains to show that the exceptional probability $p_n(x)$ is negligible relative to the lower bound in \eqref{eq:shifted-mart-tail}. From \eqref{eq:W-tail},
\begin{equation}\label{eq:p-bound}
 p_n(x)\le C\exp\{-cK(1+\log n+x^2)\}.
\end{equation}
Mills' lower bound gives
\[
 \bPhi(x)\ge \frac{c}{1+x}e^{-x^2/2}.
\]
Also, {uniformly for $0\le x\le b_n$,
\begin{equation}\label{eq:Delta-coarse}
 \Delta_n(x)\le \eta_nx^2+C\log n,
 \qquad \eta_n\longrightarrow0.
\end{equation}
To see this, take
\[
 \eta_n:=\frac{b_n}{\sqrt n}+\frac{2\log n}{\sqrt n},
\]
which tends to zero.  If $1\le x\le b_n$, then
\[
 \frac{x^3}{\sqrt n}\le\frac{b_n}{\sqrt n}x^2,
 \qquad
 \frac{(1+x)\log n}{\sqrt n}
 \le\frac{2x\log n}{\sqrt n}
 \le\frac{2\log n}{\sqrt n}x^2.
\]
If $0\le x<1$, then directly $\Delta_n(x)\le C(\log n)/\sqrt n\le C\log n$, proving \eqref{eq:Delta-coarse}.

Let $C$ now denote the fixed constant in the lower bound of \eqref{eq:shifted-mart-tail}.  Mills' lower bound and \eqref{eq:Delta-coarse} imply
\begin{align*}
 \bPhi(x)e^{-C\Delta_n(x)}
 &\ge \frac{c}{1+x}
 \exp\left\{-\frac{x^2}{2}-C\eta_nx^2-C_1\log n\right\}\\
 &\ge c'\exp\{-A x^2-B\log n\}
\end{align*}
for all sufficiently large $n$, where $A,B>0$ are absolute constants, here we used $\eta_n\le1$ and $\log(1+x)\le x^2/2+\log2$.  On the other hand, writing the constants in \eqref{eq:p-bound} as $C_p,c_p>0$,
\[
 p_n(x)\le C_p\exp\{-c_pKx^2-c_pK\log n-c_pK\}.
\]
Choose $K$ so large that $c_pK>A+1$ and $c_pK>B+3$.  The fixed multiplicative constants are then absorbed for all sufficiently large $n$, and we obtain
\begin{equation}\label{eq:p-negligible}
 p_n(x)\le n^{-2}\bPhi(x)e^{-C\Delta_n(x)}
\end{equation}
uniformly in the moderate-tail range.}

The upper bound in \eqref{eq:sandwich}, \eqref{eq:shifted-mart-tail}, and \eqref{eq:p-negligible} imply
\[
 \Pp(Y_n>x)\le \bPhi(x)e^{C\Delta_n(x)}
 \left\{1+n^{-2}e^{-2C\Delta_n(x)}\right\}
 \le \bPhi(x)e^{C_2\Delta_n(x)}.
\]
For the lower bound,
\[
 \Pp(Y_n>x)\ge \bPhi(x)e^{-C\Delta_n(x)}(1-n^{-2}),
\]
and $|\log(1-n^{-2})|\le Cn^{-2}\le C\Delta_n(x)$. This proves \eqref{eq:main-upper}. Applying the same argument to
\[
 -Y_n=\alpha_n(-Z_n)-W_n
\]
and using the lower-tail estimate in Proposition~\ref{prop:martingale-results} proves \eqref{eq:main-lower}.

Finally, if $x>0$, then
\[
 \Pp(Y_n\le-x)
 =\lim_{m\to\infty}\Pp\left(Y_n<-x+\frac1m\right).
\]
Apply \eqref{eq:main-lower} at $x-1/m$ and let $m\to\infty$. At $x=0$, write
\[
 \frac{\Pp(Y_n\le0)}{\bPhi(0)}
 =2-\frac{\Pp(Y_n>0)}{\bPhi(0)}.
\]
The ratio on the right is $1+O((\log n)/\sqrt n)$ by \eqref{eq:main-upper}, hence the left-hand ratio has the same form, and its logarithm is bounded by $C\Delta_n(0)$. Thus the non-strict lower-tail version also holds.
\end{proof}

Once the Cram\'er-type moderate deviation estimates for $Y_n$ are available, the moderate deviation principle and Berry--Esseen bound follow naturally, see Fan and Shao \cite{FanShao2024}. For the convenience of readers, we give the details here.

\begin{proof}[Proof of Corollary~\ref{cor:MDP}]
Set $X_n:=Y_n/a_n=(\log q_n-n)/(a_n\sqrt n)$. For every fixed $y>0$, the condition $a_n=o(\sqrt n)$ allows us to apply Theorem~\ref{thm:main-cramer} at $x=a_ny$. We obtain
\[
 \log\Pp(X_n>y)
 =\log\bPhi(a_ny)+O\{\Delta_n(a_ny)\}.
\]
Moreover,
\begin{align*}
 \frac{\Delta_n(a_ny)}{a_n^2}
 &\le C_y\left\{
 \frac{a_n}{\sqrt n}
 +\frac{\log n}{a_n\sqrt n}
 +\frac{\log n}{a_n^2\sqrt n}
 \right\}\longrightarrow0,
\end{align*}
while the Gaussian tail asymptotic implies
\[
 \frac1{a_n^2}\log\bPhi(a_ny)\longrightarrow-\frac{y^2}{2}.
\]
Hence
\begin{equation}\label{eq:MDP-tail}
 \lim_{n\to\infty}\frac1{a_n^2}\log\Pp(X_n>y)
 =-\frac{y^2}{2},
 \qquad y>0.
\end{equation}
The lower-tail part of Theorem~\ref{thm:main-cramer} similarly gives
\begin{equation}\label{eq:MDP-tail-left}
 \lim_{n\to\infty}\frac1{a_n^2}\log\Pp(X_n<-y)
 =-\frac{y^2}{2},
 \qquad y>0.
\end{equation}

We now derive the full large-deviation bounds at speed $a_n^2$. First, \eqref{eq:MDP-tail}--\eqref{eq:MDP-tail-left} imply exponential tightness: for every $L>0$,
\[
 \limsup_{n\to\infty}\frac1{a_n^2}
 \log\Pp(|X_n|>L)\le-\frac{L^2}{2},
\]
and the right-hand side tends to $-\infty$ as $L\to\infty$.

Let $F\subset\R$ be closed. If $0\in F$, then
\[
 \limsup_{n\to\infty}\frac1{a_n^2}\log\Pp(X_n\in F)
 \le0=-\inf_{x\in F}\frac{x^2}{2}.
\]
If $0\notin F$, let $r=\inf_{x\in F}|x|>0$. For any $0<\eta<r$,
\[
 F\subset(-\infty,-r]\cup[r,\infty)
 \subset\{x:|x|>r-\eta\}.
\]
Using \eqref{eq:MDP-tail}--\eqref{eq:MDP-tail-left} and the elementary rule for the logarithm of a sum yields
\[
 \limsup_{n\to\infty}\frac1{a_n^2}\log\Pp(X_n\in F)
 \le-\frac{(r-\eta)^2}{2}.
\]
Letting $\eta\downarrow0$ gives the required upper bound.

Let $G\subset\R$ be open and choose $y\in G$. If $y>0$, select $\eta\in(0,y)$ so small that $[y-\eta,y+\eta]\subset G$. Then
\[
 \Pp(X_n\in G)
 \ge \Pp(X_n>y-\eta)-\Pp(X_n>y+\eta).
\]
By \eqref{eq:MDP-tail}, the second probability is exponentially negligible relative to the first, because $(y+\eta)^2>(y-\eta)^2$. Therefore
\[
 \liminf_{n\to\infty}\frac1{a_n^2}\log\Pp(X_n\in G)
 \ge-\frac{(y-\eta)^2}{2}.
\]
Letting $\eta\downarrow0$ gives the lower bound $-y^2/2$. The case $y<0$ follows from \eqref{eq:MDP-tail-left}. If $y=0$, choose $\eta>0$ with $(-\eta,\eta)\subset G$. Equations \eqref{eq:MDP-tail} and \eqref{eq:MDP-tail-left} imply $\Pp(|X_n|\ge\eta)\to0$, and hence
\[
 \liminf_{n\to\infty}\frac1{a_n^2}\log\Pp(X_n\in G)=0.
\]
Taking the supremum over $y\in G$ proves the lower bound and completes the MDP proof.
\end{proof}

\begin{proof}[Proof of Corollary~\ref{cor:BE}]
Let $x_n=2\sqrt{\log n}$. Then $x_n=o(n^{1/6})$ and $\Delta_n(x_n)\to0$. Therefore Theorem~\ref{thm:main-cramer} and $|e^u-1|\le C|u|$ for $|u|\le1$ imply, uniformly for $0\le x\le x_n$,
\begin{align*}
 |\Pp(Y_n>x)-\bPhi(x)|
 &\le C\bPhi(x)\Delta_n(x)\\
 &\le \frac{C}{\sqrt n}
 \left\{x^3\bPhi(x)+(1+x)\bPhi(x)\log n\right\}\\
 &\le C\frac{\log n}{\sqrt n},
\end{align*}
because $\sup_{x\ge0}x^3\bPhi(x)<\infty$ and $\sup_{x\ge0}(1+x)\bPhi(x)<\infty$. If $x>x_n$, monotonicity and the main theorem at $x_n$ give
\[
 \Pp(Y_n>x)+\bPhi(x)
 \le \Pp(Y_n>x_n)+\bPhi(x_n)
 \le C\bPhi(x_n)
 \le Cn^{-2}.
\]
Thus, for every $x\ge0$,
\[
 |\Pp(Y_n\le x)-\Phi(x)|
 =|\Pp(Y_n>x)-\bPhi(x)|
 \le C\frac{\log n}{\sqrt n}.
\]

For $x<0$, put $t=-x>0$. The non-strict lower-tail estimate in Theorem~\ref{thm:main-cramer} gives the same bounds for
\[
 |\Pp(Y_n\le -t)-\bPhi(t)|
 =|\Pp(Y_n\le x)-\Phi(x)|.
\]
Combining the two half-lines proves \eqref{eq:BE-main} for all sufficiently large $n$. Enlarging $C$ covers the finitely many remaining values $3\le n<n_0$.
\end{proof}

\appendix
\section{Cram\'er-type moderate deviations for martingales}\label{app:FanShao}

For reference, we record the two consequences of Fan and Shao
\cite{FanShao2024} used in the proof. For each $n$, let
$(\xi_{i,n},\mathcal G_{i,n})_{1\le i\le n}$ be a finite
martingale-difference array. Set
\[
 S_n:=\sum_{i=1}^n\xi_{i,n},
 \qquad
 \langle S\rangle_n:=\sum_{i=1}^n
 \E(\xi_{i,n}^2\mid\mathcal G_{i-1,n}).
\]
Assume that there are numbers $\epsilon_n,\delta_n\in(0,1/2]$, with
$\epsilon_n,\delta_n\to0$, and an absolute constant $C_0$ such that
\begin{equation}\label{eq:FS-A1}
 \big|\E(\xi_{i,n}^m\mid\mathcal G_{i-1,n})\big|
 \le\frac{m!}{2}\epsilon_n^{m-2}
 \E(\xi_{i,n}^2\mid\mathcal G_{i-1,n}),
 \qquad m\ge2,\quad 1\le i\le n,
\end{equation}
and
\begin{equation}\label{eq:FS-A2}
 \Pp(|\langle S\rangle_n-1|\ge u)
 \le C_0\exp\{-u^2\delta_n^{-2}\},
 \qquad u>0.
\end{equation}

\begin{lemma}\label{lem:FS-cramer}
Under \eqref{eq:FS-A1}--\eqref{eq:FS-A2}, there is a constant $C>0$, depending only on $C_0$, such that, uniformly for
\[
 0\le x=o\!\left(\min\{\epsilon_n^{-1},\delta_n^{-1}\}\right),
\]
\begin{equation}\label{eq:FS-cramer}
 \left|\log\frac{\Pp(S_n>x)}{1-\Phi(x)}\right|
 \le C\left[
 x^3(\epsilon_n+\delta_n)
 +(1+x)\{\delta_n|\log\delta_n|
 +\epsilon_n|\log\epsilon_n|\}
 \right].
\end{equation}
The same conclusion holds with $S_n$ replaced by $-S_n$.
\end{lemma}

The above lemma is Theorem~2.2 of Fan and Shao \cite{FanShao2024}. The lower-tail assertion follows by applying the theorem to $-\xi_{i,n}$. The following lemma is Theorem~4.1 of Fan and Shao \cite{FanShao2024}.

\begin{lemma}\label{lem:FS-BE}
Under \eqref{eq:FS-A1}--\eqref{eq:FS-A2}, there is a constant $C>0$, depending only on $C_0$, such that
\begin{equation}\label{eq:FS-BE}
 \sup_{x\in\R}|\Pp(S_n\le x)-\Phi(x)|
 \le C\{\epsilon_n|\log\epsilon_n|
 +\delta_n|\log\delta_n|\}.
\end{equation}
\end{lemma}

\end{document}